\documentclass{scrartcl}

\usepackage[english]{babel}
\usepackage[utf8]{inputenc}
\usepackage{mathtools}
\usepackage{mathpazo}
\usepackage{amsmath}
\usepackage{amssymb}
\usepackage{amsthm}
\usepackage{graphicx}
\usepackage[caption=false]{subfig}
\usepackage[top=1.25in]{geometry}
\usepackage{enumitem}
\usepackage{todonotes}
\setenumerate{label=\textup{(\alph*)}}

\usepackage[hyphens]{url}
\usepackage[hidelinks]{hyperref}
\usepackage{cleveref}
\usepackage{mathrsfs}

\newtheorem{lemma}{Lemma}[section]
\newtheorem{theorem}[lemma]{Theorem}
\newtheorem{proposition}[lemma]{Proposition}

\theoremstyle{definition}
\newtheorem{remark}[lemma]{Remark}

\newtheorem{example}[lemma]{Example}
\newtheorem{assumption}{Assumption}

\makeatletter
\DeclareOldFontCommand{\sc}{\normalfont\scshape}{\@nomath\sc}
\makeatother

\crefdefaultlabelformat{#2\textup{#1}#3}
\Crefname{assumption}{Assumption}{Assumptions}
\Crefrangeformat{assumption}{Assumptions #3\textup{#1}#4--#5\textup{#2}#6}

\newcommand{\keywords}[1]
{
{\small
\textbf{Key words.} {#1}
\\
}
}

\newcommand{\amssubject}[1]
{
{\small
\textbf{AMS subject classifications.} {#1}
}
}

\renewcommand{\abstract}[1]
{
{\small
\textbf{Abstract.} {#1}
\\
}
}

\title{Convergence rates for ensemble-based
solutions to optimal control of uncertain
dynamical systems\footnote{\textbf{Funding}: This material is based upon work supported by the National Science Foundation under
Award No.\ DMS-2410944.}}

\author{Olena Melnikov\footnotemark[2] \and Johannes Milz\thanks{H.\ Milton Stewart School of Industrial and Systems Engineering, Georgia Institute of Technology, Atlanta, Georgia 30332 (omelnikov6@gatech.edu, johannes.milz@isye.gatech.edu)}}
\date{February 3, 2026}

\begin{document}

\maketitle

\abstract{%
We consider optimal control problems involving nonlinear ordinary differential equations with uncertain inputs. Using the sample average approximation, we obtain optimal control problems with ensembles of deterministic dynamical systems. Leveraging techniques for metric entropy bounds, we derive non-asymptotic Monte Carlo-type convergence rates for the ensemble-based solutions. Our theoretical framework is validated through numerical simulations on a harmonic oscillator problem and a vaccination scheduling problem for epidemic control under model parameter uncertainty.
}


\amssubject{90C15, 90C30, 49K45, 49N60}

\section{Introduction}
\label{sec:introduction}
We consider the optimal control problem
\begin{align}
\label{eq:optimalcontrolproblem}
\min_{u \in L^2(0,1;\mathbb{R}^m)} \,
\mathbb{E}[F(x^u(1,\xi), \xi)] + \psi(u),
\end{align}
where for each
parameter $\xi \in \Xi$
and control $u(\cdot) \in L^2(0,1;\mathbb{R}^m)$,
$x^u(\cdot, \xi) = x(\cdot, \xi)$ solves the parameterized
affine-control dynamical system
\begin{align}
\label{eq:uncertain-ode}
 \dot{x}(t, \xi)  = f_0(x(t,\xi),\xi) +  f_1(x(t,\xi),\xi) u(t)
 \quad \text{for a.e. }  t \in (0,1),
 \quad x(0,\xi) = x_0(\xi),
\end{align}
where $\Xi$ is a complete separable metric space
equipped with its Borel sigma-algebra,
$F \colon \mathbb{R}^n \times \Xi
\to [0,\infty)$,
$f_0 \colon \mathbb{R}^n \times \Xi \to\mathbb{R}^n$
and
$f_1 \colon \mathbb{R}^n \times \Xi \to\mathbb{R}^{n\times m}$
are Carath\'eodory mappings,
and $x_0 \colon \Xi \to \mathbb{R}^n$ is measurable.
The function $\psi \colon L^2(0,1;\mathbb{R}^m) \to
(-\infty,\infty]$ is proper, lower semicontinuous, and strongly convex with parameter
$\alpha > 0$.
The parameterized initial value
problem in \eqref{eq:uncertain-ode} models
uncertain right-hand sides and initial values.
With some abuse of notation, we use $\xi$ to denote elements of $\Xi$
and a random element taking values in $\Xi$.

We use the sample average approximation (SAA) approach
\cite{Kleywegt2002} to approximate the infinite dimensional
optimization problem \eqref{eq:optimalcontrolproblem}.
Throughout the text,
let $\xi^1, \xi^2, \ldots $ be independent and identically
distributed  $\Xi$-valued random elements defined on a complete
probability space $(\Omega, \mathcal{F}, P)$ such that each $\xi^i$ has the same distribution
as $\xi$. We obtain the SAA problem
\begin{align}
\label{eq:saaproblem}
\min_{u \in L^2(0,1;\mathbb{R}^m)} \,
\frac{1}{N} \sum_{i=1}^N F(x^u(1,\xi^i(\omega)), \xi^i(\omega)) + \psi(u).
\end{align}

The optimization problem \eqref{eq:saaproblem}
is an optimal control problem with
an ensemble of $N$  dynamical systems.
For the numerical solution of  \eqref{eq:saaproblem}, we generate
a realization of $\xi^1(\omega), \xi^2(\omega), \ldots, \xi^N(\omega)$
and solve \eqref{eq:saaproblem} for these fixed elements.
In the remainder of the text, we often omit explicit evaluations
of the random elements $\xi^i$ at $\omega \in \Omega$.

\subsection{Contributions}

In this section, we summarize our four main contributions. Our summary uses
standard notation, which is
formally introduced in \cref{sect:notation}.
We define the parameterized integrand
\begin{align*}
T(u,\xi) \coloneqq F(x^u(1,\xi), \xi).
\end{align*}
Furthermore, we define
$\psi_\alpha(u) \coloneqq \psi(u) -  (\alpha/2)\|u\|_{L^2(0,1;\mathbb{R}^m)}^2$,
\begin{align}
\label{eq:objective-functions}
g(u) \coloneqq \mathbb{E}[T(u,\xi)],
\quad
\text{and} \quad
\widehat{g}_N(u) \coloneqq \frac{1}{N} \sum_{i=1}^N
T(u,\xi^i).
\end{align}

The following items list our four main contributions.

\begin{enumerate}[nosep,wide,label=\textbf{(\alph*)}]
\item
\label{contribution:a}
We establish nonasymptotic mean convergence rates for the
SAA optimal values. Specifically, we show that
for all $N \in \mathbb{N}$,
\begin{align}
\label{eq:objectiveconvergencerate}
\mathbb{E}[|\widehat{v}_N^*-v^*|]
\leq
\frac{\text{Const}}{\sqrt{N}}
\bigg(1 + \frac{1}{\sqrt{\alpha}}\bigg),
\end{align}
where  $v^*$ is the optimal value of  \eqref{eq:optimalcontrolproblem}
and $\widehat v_N^*$ is that of \eqref{eq:saaproblem}.
Moreover, $\text{Const}$ is a constant that does not depend
on the sample size $N$ nor on the strong convexity parameter
$\alpha$. However, it can depend on other problem data,
such as the control's dimension $m$.
\item
\label{contribution:b}
We demonstrate nonasymptotic mean convergence rates
for a criticality measure for \eqref{eq:optimalcontrolproblem}
evaluated at SAA critical points:
for each critical point $u_N^* \in \mathrm{dom}(\psi)$
of \eqref{eq:saaproblem},
that is, $\widehat{\chi}_N(u_N^*)  = 0$, we show that
for all $N \in \mathbb{N}$,
\begin{align}
\label{eq:convergencerate}
\mathbb{E}[\chi(u_N^*)] \leq \frac{\text{Const}}{\sqrt{N}}
\bigg(1 + \frac{1}{\sqrt{\alpha}}\bigg),
\end{align}
where $\mathrm{Const}$
is as in \eqref{eq:objectiveconvergencerate} and
the criticality measures $\chi$ and $\widehat{\chi}_N$
 are defined by
\begin{align}
\label{eq:truecriticalitymeausre}
\chi(u) \coloneqq
\|u-
\mathrm{prox}_{\psi_\alpha}(u-\nabla g(u) -\alpha  u)\|_{L^2(0,1;\mathbb{R}^m)}
\end{align}
and
\begin{align*}
\widehat{\chi}_N(u) \coloneqq
\|u-
\mathrm{prox}_{\psi_\alpha}(u-\nabla
\widehat{g}_N(u)-\alpha  u)
\|_{L^2(0,1;\mathbb{R}^m)}.
\end{align*}

\item
\label{contribution:c}
Our contributions in \ref{contribution:a} and \ref{contribution:b} are derived using uniform
expectation bounds for Hilbert space-valued
``sub-Gaussian-type'' Carath\'eodory mappings.
Since this result may be of independent interest, we summarize it here.
Let $H$ be a real, separable Hilbert space,
let $C \subset H$ be a nonempty, closed set
with diameter $D > 0$,
and let $G \colon C \times \Xi \to H$ be a Carath\'eodory mapping
such that $\xi \mapsto G(x,\xi)$ is integrable and $\mathbb{E}[G(x,\xi)] = 0$
for each $x \in C$.
Furthermore, suppose that there exists a constant $M > 0$ such that
the sub-Gaussian-type estimate
\begin{align*}
\mathbb{E}[\cosh(\lambda \|G(x,\xi) - G(y,\xi)\|_H)] \leq
\exp(M^2 \lambda^2 \|x-y\|_H^2/2)
\; \text{for all} \;x, y \in C,
\; \lambda \geq 0
\end{align*}
holds.
Under these technical conditions, we establish the uniform
 expectation bound
\begin{align}
\label{eq:metricentropybound}
\mathbb{E}[\sup_{x \in C}\,
\|\widehat{G}_N(x)\|_H]
\leq \mathbb{E}[\|\widehat G_N(x_0)\|_H]
+  \frac{4 \sqrt{3} M }{\sqrt{N}} \int_{0}^{D/2}
\sqrt{\ln(2\mathcal{N}(C, \varepsilon))}
\, \mathrm{d}\varepsilon,
\end{align}
where
$\widehat{G}_N(x)
\coloneqq (1/N)\sum_{i=1}^N
G(x,\xi^i)$, and
$\mathcal{N}(C,\varepsilon)$ is the
$\varepsilon$-covering number of $C$
(see \cref{sect:notation} for a definition).
This uniform expectation bound is valid
for each fixed $x_0 \in C$.

\item Contributions \ref{contribution:a} through \ref{contribution:c} are technical in nature. Our final contribution involves formulating vaccination scheduling for epidemic control with random inputs as a risk-neutral optimal control problem. We validate our theoretical findings for
this vaccination scheduling problem through numerical simulations.
\end{enumerate}

\subsection{Related work}

The SAA approach has been used for various optimal control problems with uncertain dynamics, notably in optimizing batch reactors
with model parameter uncertainty \cite{Ruppen1995,Srinivasan2003}. Recent texts have established some of its theoretical properties, including asymptotic consistency and sample complexity. The paper \cite{Phelps2016} analyzes the almost sure epiconvergence of the SAA objective function and consistency statements for SAA critical points.
Further asymptotic properties are demonstrated in \cite{Scagliotti2023}.
Moreover, the work \cite{Milz2021} establishes the optimal sample complexity of
strongly convex stochastic optimization in Hilbert space.

Additionally, theoretical analyses
of the SAA framework have been extended to
other  infinite dimensional stochastic
nonconvex optimization problems with differential
equation constraints. This includes
the derivation of sample size estimates
for semilinear PDE-constrained optimization
under uncertainty \cite{Milz2022b} and
the consistency analysis of
sample-based approximations to risk-averse stochastic optimization
problem with infinite dimensional decision spaces
\cite{Milz2022a}.
For risk-neutral linear elliptic PDE-constrained optimization,
mean convergence rates for SAA optimal values and
solutions, as well as a central limit theorem, are established
in \cite{Roemisch2021}.
Parts of our theoretical analysis is based on the construction
of a deterministic compact set containing the SAA critical points.
This construction is inspired by the analysis
developed in \cite{Milz2022b}, but we use
different statistical tools to derive convergence rates.

The right-hand side of \eqref{eq:metricentropybound} resembles the classical metric entropy bound for the suprema of sub-Gaussian processes (see, for example, \cite[Thm.\ 3.1 on p.\ 95]{Buldygin2000}).
Metric entropy bounds are one of the tools
for analyzing the SAA approach; see, for example, \cite{Lew2024}.

\subsection{Manuscript overview}

We present notation and collect basic terminology in \cref{sect:notation}.
In \cref{sec:assumptions}, we formulate technical assumptions on the optimal control problem and the initial value problem.
These assumptions enable us to establish the gradient and control regularity in
\cref{sec:gradientregularity,sec:controlregularity}. Our main
results on the convergence rates for the SAA optimal values and critical points
are presented in \cref{sec:ratesvalues,sec:ratescritialponts}.
These theoretical results are empirically validated
in \cref{sec:numericalillustrations}
through two examples. Our appendices collect and establish results necessary for establishing the SAA convergence rates. While \cref{sec:optimality-conditions}
states essentially known first-order optimality conditions, \cref{sec:coveringnumbers} derives bounds
on
the covering numbers of square-integrable
mappings with square-integrable derivatives.
Finally, \cref{section:uniformexpectation} concludes by establishing the uniform exponential tail bounds.

\section{Notation}
\label{sect:notation}

We collect basic notation and terminology.
If not stated otherwise, equations involving random elements hold with
probability one.
Let $H$ be a Hilbert space with inner product $(\cdot, \cdot)_H$.
We use $\|\cdot\|_H$ to denote the norm induced by its inner
product $(\cdot, \cdot)_H$.
For a Banach space $X$, $\|\cdot\|_X$ denotes its norm.
For Banach spaces $X_1$, $\ldots$, $X_m$, we
equip the Cartesian product  $X_1 \times \cdots \times X_m$
with the norm
$\|(x_1, \ldots, x_m)\|_{X_1 \times \cdots \times X_m} \coloneqq
 (\sum_{i=1}^m\|x_i\|_{X_i}^2)^{1/2}
$
if $X_1$, $\ldots$, $X_m$ are Hilbert spaces,
and
$\|(x_1, \ldots, x_m)\|_{X_1 \times \cdots \times X_m} \coloneqq
\sum_{i=1}^m\|x_i\|_{X_i}
$ otherwise.
We denote by
$\mathbb{B}_{X}(0;r)$
 the open unit norm
ball with radius $r$ and by $\overline{\mathbb{B}}_{X}(0;r)$
its closure.
If not stated otherwise, we equip $\mathbb{R}^m$ with the
Euclidean norm $\|\cdot\|_2$.
We denote the spectral norm of a matrix $A$
and more generally
the operator norm of a linear operator $A$ between Euclidean spaces by $\|A\|_2$.
We call an extended real-valued function $G$
defined on a Hilbert space $H$ strongly convex
with parameter $\alpha > 0$
if $G(\cdot) - (\alpha/2) \|\cdot\|_{H}^2$
is convex.
The domain of a function $\phi$ is denoted by
$\mathrm{dom}(\phi)$.
For a proper, lower semicontinuous, convex function $\phi$ on a Hilbert space $H$,
its proximity operator is defined by
$
\mathrm{prox}_{\phi}(x) \coloneqq
 \mathrm{argmin}_{y \in H}\, (1/2)\|x-y\|_H^2 + \phi(y)
$.

The spaces $L^p(0,1;\mathbb{R}^m)$
$(1 \leq p \leq \infty)$ denote the ``usual'' Lebesgue spaces
equipped with their canonical norms.
For $1 \leq p \leq \infty$,
we define the Sobolev space
\begin{align*}
W_1^p(0,1;\mathbb{R}^m)
\coloneqq \{\, v \in L^p(0,1;\mathbb{R}^m) \colon \,
 \dot{v} \in L^p(0,1;\mathbb{R}^m) \, \}.
\end{align*}
Here, $\dot{v}$ denotes the
distributional (time) derivative of $v$. While we generally use the ``dot notation,'' we occasionally write $D_t v$ in contexts where $\dot{v}$ is typographically inconvenient.
We equip $W_1^p(0,1;\mathbb{R}^m)$
with
$$
\|v\|_{W_1^p(0,1;\mathbb{R}^m)}
\coloneqq
\big(\|v\|_{L^p(0,1;\mathbb{R}^m)}^p
+
\|\dot{v}\|_{L^p(0,1;\mathbb{R}^m)}^p
\big)^{1/p}
$$
if  $1 \leq p < \infty$ and
$$
\|v\|_{W_1^\infty(0,1;\mathbb{R}^m)}
\coloneqq
\|v\|_{L^\infty(0,1;\mathbb{R}^m)}
+
\|\dot{v}\|_{L^\infty(0,1;\mathbb{R}^m)}
$$
if $p = \infty$.

Let $H_2 \subset H_1$ be convex sets in a real Hilbert space $H$.
Adapting Definition~5.1.26 in \cite{Polak1997},
we say  that $G \colon H_1 \to \mathbb{R}$ is Gateaux differentiable
relative to $H_1$ at a point $v\in H_2$ if there exists
a linear bounded operator $DG(v) $ on $H$
such that for each $h \in H$ with $v+ \lambda h \in H_1$
for some $\lambda >0$,
$G$ is directionally differentiable at $v$ in the direction
$h$ with directional derivative $DG(v)h$.
We say that $G$ is Gateaux differentiable on $H_2$
relative to $H_1$ if it is Gateaux differentiable
relative to $H_1$
at each $v\in H_2$.
We denote by $\nabla G(v)$ the Riesz representation of
$DG(v)$.
(Different from Definition~5.1.26 in \cite{Polak1997}, we do not
require $G$ be continuous).
We denote the Gateaux derivative of a function
$G$ by $DG$. The partial derivative of $G$ with respect to $v$ is denoted by $D_v G$.

For a metric space $(Y, d_Y)$ and a nonempty, totally bounded subset $Y_0 \subset Y$, we define the $\nu$-covering number
$\mathcal{N}(Y_0, Y,\nu)$
as the minimal number of points in a $\nu$-net
of $Y_0$ in $Y$.
We note that the elements in  $\nu$-nets of $Y_0$ in $Y$ are not required to be contained in $Y_0$.
Since $Y_0$ is totally bounded, $(Y_0,d_Y)$ is a totally bounded metric space. For abbreviation, we define
$\mathcal{N}(Y_0, \nu) \coloneqq \mathcal{N}(Y_0, Y_0; \nu)$.

\section{Assumptions on optimal control problem}
\label{sec:assumptions}

We state a number of assumptions on the
optimal control problem \eqref{eq:optimalcontrolproblem} and
the initial value problem \eqref{eq:uncertain-ode}, in addition to those stated in \cref{sec:introduction}.
The assumptions ensure, for example, well-posedness of
the optimal control problems,
differentiability of sample average and expectation functions,
and existence and uniqueness of solutions to the parameterized
initial value problem. Moreover, these conditions allow us
to demonstrate that the integrand in \eqref{eq:optimalcontrolproblem}
has square integrable time derivatives for each parameter $\xi$.
These conditions are mainly inspired by those used in
\cite{Phelps2016,Polak1997}. After formulating the first
assumption, we comment on three key differences of our
problem formulation and that of the constrained risk-neutral control problem in \cite{Phelps2016}.

Our analysis is based on the following conditions on
the control regularizer.

\begin{assumption}
\label{assumption:feasibleset-regularizer}
There exists a constant $r_{\psi} \in (0, \infty)$ such that
$\|u\|_{L^\infty(0,1;\mathbb{R}^m)} \leq r_{\psi} $
for all $u \in \mathrm{dom}(\psi)$.
\end{assumption}

\Cref{assumption:feasibleset-regularizer} ensures that the
domain of $\psi$ is a bounded subset in $L^\infty(0,1;\mathbb{R}^m)$.
\Cref{assumption:feasibleset-regularizer} is, for example, satisfied
for the sparsity-promoting control regularizer
$$\psi(u) = (\alpha/2) \|u\|_{L^2(0,1;\mathbb{R}^m)}^2
+ \beta  \|u\|_{L^1(0,1;\mathbb{R}^m)}
+  \iota_{ \{u \in L^2(0,1;\mathbb{R}^m) \colon
\|u\|_{L^\infty(0,1;\mathbb{R}^m)} \leq 1
\}}(u),$$
where $\alpha > 0$, and $\beta \geq 0$.
Moreover, $\iota_Y \colon L^2(0,1;\mathbb{R}^m) \to \{0,\infty\}$
is the indicator function of
the set $Y \subset L^2(0,1;\mathbb{R}^m)$,
that is, $\iota_Y(y) = 0$ if $y \in Y$ and $\iota_Y(y) = \infty$
otherwise.

Our formulation of the control problem \eqref{eq:optimalcontrolproblem} differs from that of the constrained risk-neutral control problem
in \cite{Phelps2016} in three key aspects:
(i) we consider affine-linear controls in \eqref{eq:uncertain-ode} instead of potentially
nonaffine control dependence, (ii) we impose the strong convexity of $\psi$, and (iii) we omit the inclusion of additional $n$-dimensional decision variables that are added to the initial value $x_0(\xi)$ in \eqref{eq:uncertain-ode}. The first two aspects enable us to construct a deterministic compact set containing the SAA solutions. By omitting the additional, finite dimensional decision variables, we can present a concise sample complexity analysis focused on the infinite dimensional controls.  Nonetheless, the inclusion of finite dimensional spaces remains a feasible direction for future work. Despite these modifications, we are still able to leverage the theoretical results from \cite{Phelps2016}, including differentiability statements and Lipschitz constant estimates of the integrand.

The next assumption is adapted from Assumption~1 in \cite{Phelps2016}.
It requires that the parameterized initial value problem has a solution
and that the state be contained in a deterministic compact set.

\begin{assumption}
\label{assumption:essentiallyboundedtrajectories}
There exists a convex, compact set $X_0 \subset \mathbb{R}^n$
with $0 \in X_0$ such that for each
$\xi \in \Xi$
and
$u \in \mathbb{B}_{L^\infty(0,1;\mathbb{R}^m)}(0;2r_\psi)$,
it holds that
$x^u(t,\xi) \in X_0$ for a.e.\ $t \in [0,1]$.
\end{assumption}

\Cref {assumption:essentiallyboundedtrajectories}
implies that the parameterized initial value is bounded.
More precisely, \Cref {assumption:essentiallyboundedtrajectories}  ensures
the existence of a constant $V_0 \in (0,\infty)$
such that $\|x_0(\xi)\|_2 \leq V_0$ for all $\xi \in \Xi$.

The next assumption is motivated by Assumption~5.6.2 in
\cite{Polak1997} and Assumptions~2, 3, and 5 in \cite{Phelps2016}.
We define the dynamical system's right-hand
side
$f \colon \mathbb{R}^n \times \mathbb{R}^m \times \Xi \to \mathbb{R}^n$
by
\begin{align}
\label{eq:rhs}
f(x,u,\xi) \coloneqq f_0(x,\xi) +  f_1(x,\xi) u.
\end{align}

\begin{assumption}
\label{assumption:dynamicalsystem}
Let $r_{\psi}$ be as in \Cref{assumption:feasibleset-regularizer},
and let $X_0$ be as in \Cref{assumption:essentiallyboundedtrajectories}.
\begin{enumerate}[nosep]
\item For each $\xi \in \Xi$,
$f_0(\cdot,\xi)$ and $f_1(\cdot,\xi)$ are continuously differentiable, and
for each $x \in X_0$,
$D_x f_0(x,\cdot)$, and $D_x f_1(x,\cdot)$
are measurable.
\item
There exists a constant $L_f\in (0,\infty)$ such that for each $\xi \in \Xi$,
$f_0(\cdot, \xi)$, $f_1(\cdot, \xi)$, $D_x f_0(\cdot,\xi)$, and $D_x f_1(\cdot,\xi)$
are Lipschitz continuous on an open neighborhood of  $X_0$
with Lipschitz constant $L_f$. Moreover,
$\|f_0(0,\xi)\|_2 \leq L_f$
for all $\xi \in \Xi$,
and
$\|f_1(x,\xi)\|_2 \leq L_f$
and
$\|D_x f_1(x,\xi)\|_2 \leq L_f$
for all $(x,\xi) \in X_0 \times \Xi$.

\item
For each $\xi \in \Xi$,
$F(\cdot,\xi)$ is continuously differentiable,
and for each $x \in X_0$,
$F(x, \cdot)$ and $D_xF(x,\cdot)$ are measurable.
\item
There exists a constant $L_F^\prime \in (0,\infty)$
such that for each $\xi \in \Xi$,
$F(\cdot, \xi)$ and $D_x F(\cdot,\xi)$ are Lipschitz continuous on
an open neighborhood of $X_0$
with Lipschitz constant $L_F^\prime$.
\end{enumerate}
\end{assumption}

\Cref{assumption:feasibleset-regularizer,assumption:dynamicalsystem,assumption:essentiallyboundedtrajectories} ensure that there exists a neighborhood $\tilde X_0$ of
$X_0 $, such that for each $\xi \in \Xi$,
$f(\cdot, \cdot, \xi)$ and $D_x f(\cdot, \cdot, \xi)$ are Lipschitz continuous
on
$\tilde X_0 \times \{u \in \mathbb{R}^m \colon \|u\|_2 \leq  2 r_{\psi}\}$
with Lipschitz constant
\begin{align*}
L_{f}^\prime \coloneqq
L_f \sqrt{(1 + 2r_\psi)^2 + 1}.
\end{align*}
Hence for all $x \in X_0$, $\xi \in \Xi$, and $u \in \mathbb{R}^m$
with $\|u\|_2 \leq 2r_\psi$,
\begin{align}
\label{eq:rhs-bound}
\begin{aligned}
\|f(x,u,\xi)\|_2 & \leq
L_f^\prime(\|x\|_2 + 2 r_{\psi}+1),
\quad \text{and} \quad
&\|D_x f(x,u,\xi)\|_2  &\leq L_f^\prime.
\end{aligned}
\end{align}

We conclude this section by discussing
\Cref{assumption:essentiallyboundedtrajectories}.
\begin{remark}
For globally Lipschitz continuous right-hand sides $f(\cdot, \cdot, \xi)$
with Lipschitz constants independent of $\xi$,
standard results, such as Proposition~5.6.5 in \cite{Polak1997}, can be
used to construct a set $X_0$ satisfying
\Cref{assumption:essentiallyboundedtrajectories}.
Following the approach in \cite{Phelps2016}, we assume the existence of such a set  $X_0$.
 This enables us to work with right-hand sides that are Lipschitz continuous only on the compact set
$X_0$,  allowing a broader class of problems to be considered.
\end{remark}

\section{Gradient regularity via adjoint and state regularity}
\label{sec:gradientregularity}

The section's main purpose is to show that the
objective function's gradient is contained in a compact
subset of the feasible set. The analysis performed is based on the results provided in \cite{Phelps2016,Polak1997}.

We recall that the control problem's integrand
$T \colon \mathbb{B}_{L^\infty(0,1;\mathbb{R}^m)}(0;2r_\psi) \times \Xi \to [0,\infty)$ is defined by
\begin{align*}
T(u,\xi) \coloneqq F(x^u(1,\xi), \xi).
\end{align*}

If \Cref{assumption:feasibleset-regularizer,assumption:dynamicalsystem,assumption:essentiallyboundedtrajectories}
hold and $\xi \in \Xi$, then
Corollary~5.6.9 in \cite{Polak1997} ensures that
$T(\cdot, \xi)$ is Gateaux differentiable
on $\mathrm{dom}(\psi)$ relative
to $\mathbb{B}_{L^\infty(0,1;\mathbb{R}^m)}(0;2r_\psi) $
with
\begin{align}
\label{eq:gradient-T}
[\nabla_u T(u,\xi)](t) = f_1(x^u(t,\xi),\xi)^T p^u(t, \xi)
\quad \text{for a.e. }  t \in (0,1),
\end{align}
where $p^u(\cdot,\xi) = p(\cdot, \xi)$
solves for a.e.\ $t \in (0,1)$
the parameterized adjoint equation
\begin{align}
\label{eq:adjoint-equation}
\dot{p}(t, \xi)  =
- D_x f(x^u(t,\xi), u(t),\xi)^T p(t,\xi),
\quad p(1,\xi) = \nabla_x F(x^u(1,\xi),\xi).
\end{align}

Building on the initial discussion, the section's main purpose is to show that
$\nabla_u T(u,\xi)$ is contained in a fixed deterministic ball in
$W_1^2(0,1; \mathbb{R}^m)$
for each $(u,\xi) \in \mathrm{dom}(\psi) \times \Xi$.
This result is formulated in the following theorem, with
an explicit expression for the ball's radius.

\begin{theorem}
\label{thm:gradientregularity}
If \Cref{assumption:feasibleset-regularizer,assumption:dynamicalsystem,assumption:essentiallyboundedtrajectories}
hold,
then for each $(u,\xi) \in \mathrm{dom}(\psi) \times \Xi$,
\begin{align*}
\nabla_u T(u,\xi) \in
\overline{\mathbb{B}}_{W_1^\infty(0,1;\mathbb{R}^m)}(0;R)
\subset
\overline{\mathbb{B}}_{W_1^2(0,1;\mathbb{R}^m)}(0;R),
\end{align*}
where
\begin{align}
\label{eq:stabilityestimateradius}
R \coloneqq
R_0
+ L_f^{\prime}
\big[ (1 + 2r_\psi + V_0) \mathrm{e}^{L_f^\prime} + 2r_\psi + 2 \big]
R_0
\quad \text{with} \quad
R_0 &\coloneqq
L_f^\prime (1+ L_F^\prime) \mathrm{e}^{L_f^\prime}.
\end{align}
\end{theorem}

We prepare our proof of \Cref{thm:gradientregularity}.
The following lemma collects basic properties of $T$ mainly taken from
\cite{Phelps2016}.
We define
\begin{align}
\label{eq:LipschitzConstantT}
L_T^\prime \coloneqq \sqrt{2} L_F^\prime (L_f^\prime +1) \mathrm{e}^{L_f^\prime +1}.
\end{align}

\begin{lemma}
\label{lem:ReducedObjectiveLipschitzGateaux}
If
\Cref{assumption:feasibleset-regularizer,assumption:dynamicalsystem,assumption:essentiallyboundedtrajectories}
hold,
then the following statements are valid.
\begin{enumerate}[nosep]

\item For each $\xi \in \Xi$,
$T(\cdot, \xi)$ is Lipschitz continuous
with Lipschitz constant
$L_T^\prime$, and  Gateaux differentiable on $\mathrm{dom}(\psi)$ relative to
$\mathbb{B}_{L^\infty(0,1;\mathbb{R}^m)}(0;2r_\psi)$.
\item
There exists a constant $L_{\nabla T}^\prime \in (0,\infty)$
such that
$\nabla_u T(\cdot, \xi)$ is Lipschitz continuous
on $\mathrm{dom}(\psi)$
with Lipschitz constant $L_{\nabla T}^\prime$
for each  $\xi \in \Xi$.
\item The mappings $T $ and
$\nabla_u T$
are Carath\'eodory mappings
on $ \mathrm{dom}(\psi) \times \Xi$.
\item The  function
$u\mapsto \mathbb{E}[T(u,\xi)]$
is Gateaux differentiable on $\mathrm{dom}(\psi)$ relative to
$\mathbb{B}_{L^\infty(0,1;\mathbb{R}^m)}(0;2r_\psi)$.
\end{enumerate}
\end{lemma}

\begin{proof}
\begin{enumerate}[wide,nosep]
\item
The Lipschitz statement
follows from the proof of Lemma~5.6.7
in
\cite{Polak1997}, since $L_{f}^\prime+1\geq 1$.
See also Lemma~3.5 in \cite{Phelps2016}.
Furthermore,
Corollary~5.6.9 in \cite{Polak1997} ensures the Gateaux differentiability statement.

\item
This statement follows from an application
of Lemma 5.2 in \cite{Phelps2016}.

\item
According to Lemma 3.5 in \cite{Phelps2016},
$T$ is a Carath\'eodory mapping.
Using Lemma 2.3 in \cite{Phelps2016}, composition rules,
and \eqref{eq:gradient-T}, we can show that
$\nabla_u T(u,\cdot)$  is measurable on $\Xi$.
See also Lemma 5.2 in \cite{Phelps2016}.

\item
An application of Lemma 4.3  in \cite{Phelps2016} yields
the assertion.
\end{enumerate}
\end{proof}

Using
chain and product rules (cf.\ \cite[p.\ 340]{Leoni2017} and \cite[Cor.\ 3.2]{Ambrosio1990}),
and \eqref{eq:gradient-T},
we find that for a.e.\ $t \in (0,1)$,
\begin{align}
\label{eq:TimeDerivativeGradientT}
D_t[\nabla_u T(u,\xi)](t)
=
D_t[ f_1(x^u(t,\xi),\xi)^T] p^u(t,\xi)
+ f_1(x^u(t,\xi),\xi)^T \dot{p}^u(t,\xi).
\end{align}
This identity allows us to establish stability estimates
on the gradient's time derivative.

\begin{lemma}
\label{lem:stateadjointestimates}
If
\Cref{assumption:feasibleset-regularizer,assumption:dynamicalsystem,assumption:essentiallyboundedtrajectories}
hold,
then the following statements hold true.
\begin{enumerate}[nosep]
\item For each $(u,\xi)\in \mathbb{B}_{L^\infty(0,1;\mathbb{R}^m)}(0;2r_\psi) \times \Xi$,
the dynamical system
\eqref{eq:uncertain-ode}
has a unique solution $x^{u}(\cdot,\xi) \in W_1^\infty(0,1;\mathbb{R}^n)$, and
for a.e.\ $t \in [0,1]$,
\begin{align}
\label{eq:statestabilityestimate}
\begin{aligned}
\|x^u(t, \xi)\|_{2}
&\leq
(1+2r_{\psi} + V_0) \mathrm{e}^{L_f^\prime },
\\
\|\dot{x}^u(t, \xi)\|_2  & \leq
L_f^\prime(\|x^u(t, \xi)\|_{2} + 2 r_{\psi}+1),
\end{aligned}
\end{align}
where we recall that \( V_0 \in (0, \infty) \) is such that \( \|x_0(\xi)\|_2 \leq V_0 \) for all \( \xi \in \Xi \).

\item For each $(u,\xi)\in \mathbb{B}_{L^\infty(0,1;\mathbb{R}^m)}(0;2r_\psi) \times \Xi$,
the adjoint system
\eqref{eq:adjoint-equation}
has a unique solution $p^{u}(\cdot,\xi) \in W_1^\infty(0,1;\mathbb{R}^n)$, and for a.e.\ $t \in [0,1]$,
\begin{align}
\label{eq:adjointstabilityestimates}
\begin{aligned}
\|p^u(t,\xi)\|_2 & \leq
(1+ L_F^\prime) \mathrm{e}^{L_f^\prime},
\\
\|\dot{p}^u(t,\xi)\|_2 &
\leq L_f^\prime \|p^u(t,\xi)\|_2.
\end{aligned}
\end{align}
\item For all $(u,\xi) \in \mathrm{dom}(\psi) \times \Xi$,
\begin{align*}
\|\nabla_u T(u,\xi)\|_{L^\infty(0,1;\mathbb{R}^m)}
\leq
L_f^\prime (1+ L_F^\prime) \mathrm{e}^{L_f^\prime}.
\end{align*}
\item For all $(u,\xi) \in \mathrm{dom}(\psi) \times \Xi$,
\begin{align*}
\|D_t[\nabla_u T(u,\xi)]\|_{L^\infty(0,1;\mathbb{R}^m)}
&\leq
(L_f^{\prime})^2 (1 + L_F^\prime) \mathrm{e}^{L_f^\prime}
\\
& \quad
\cdot
\big[ (1 + 2r_\psi + V_0) \mathrm{e}^{L_f^\prime} + 2r_\psi + 2 \big].
\end{align*}
\end{enumerate}
\end{lemma}
\begin{proof}
\begin{enumerate}[nosep,wide]
\item The existence and uniqueness follow from
Proposition 5.6.5 in \cite{Polak1997}.
Using \eqref{eq:rhs-bound}, we obtain
for all $t \in [0, 1]$,
\begin{align*}
\|x^u(t, \xi)\|_2
\leq \|x_0(\xi)\|_2 + L_f^\prime
\int_{0}^{t}
\big(\|x^u(\tau, \xi)\|_2 + 2 r_{\psi}+1) \,
\mathrm{d}\tau.
\end{align*}
Defining $y(t) \coloneqq \|x^u(t,\xi)\|_2 + 2r_{\psi} +1$,
this estimate ensures  $y(t) \leq y(0)  + L_f^\prime \int_{0}^{t} y(\tau) \, \mathrm{d}\tau$ for all $t \in [0,1]$.
Now, applying Gronwall's inequality yields the first stability estimate.
Using \eqref{eq:uncertain-ode},
and \eqref{eq:rhs-bound}, we obtain the second stability estimate.
\item The existence and uniqueness of the solution to the adjoint
equation follows from Corollary~5.6.9 in \cite{Polak1997}.

The mean value theorem ensures
$\|\nabla_x F(x,\xi)\|_2 \leq L_F^\prime$
for all $(x,\xi) \in X_0 \times \Xi$.
Combined with \eqref{eq:rhs-bound}, and
Exercise~5.6.6  in \cite{Polak1997},
we obtain the first adjoint state stability estimate.
The second is a consequence of the adjoint equation
\eqref{eq:adjoint-equation}.
\item Using \eqref{eq:rhs-bound},  and
 \eqref{eq:gradient-T}, we have
 for all $(u,\xi) \in \mathrm{dom}(\psi) \times \Xi$
 and a.e.\ $t \in (0,1)$,
\begin{align*}
\|[\nabla_u T(u,\xi)](t)\|_2 \leq
\|f_1(x^u(t,\xi),\xi)\|_2 \| p^u(t, \xi)\|_2
\leq L_f^\prime   \| p^u(t, \xi)\|_2.
\end{align*}
Combined with \eqref{eq:adjointstabilityestimates},
we obtain the estimate.

\item Using \eqref{eq:rhs-bound}
and \eqref{eq:TimeDerivativeGradientT}, we obtain
\begin{align*}
\|D_t[\nabla_u T(u,\xi)]\|_{L^\infty(0,1;\mathbb{R}^m)}
&\leq
L_f^\prime\|p^u(\cdot,\xi)\|_{L^\infty(0,1;\mathbb{R}^{n})}
\|\dot{x}^u(\cdot,\xi)\|_{L^\infty(0,1;\mathbb{R}^{n})}
\\
& \quad +
L_f^\prime \|\dot{p}^u(\cdot,\xi)\|_{L^\infty(0,1;\mathbb{R}^{n})}.
\end{align*}
Now,  \eqref{eq:statestabilityestimate} and
\eqref{eq:adjointstabilityestimates} ensure the bound.

\end{enumerate}
\end{proof}

The proof of \Cref{thm:gradientregularity} is now
a direct consequence of \Cref{lem:stateadjointestimates}.

\begin{proof}[{Proof of \Cref{thm:gradientregularity}}]
Since
\begin{align*}
\|\nabla_u T(u,\xi)\|_{W_1^\infty(0,1;\mathbb{R}^m)} =
\|\nabla_u T(u,\xi)\|_{L^\infty(0,1;\mathbb{R}^m)}+
\|D_t[\nabla_u T(u,\xi)]\|_{L^\infty(0,1;\mathbb{R}^m)},
\end{align*}
\Cref{lem:stateadjointestimates} ensures the assertion.
\end{proof}

\section{Control regularity via gradient regularity}
\label{sec:controlregularity}

We show that SAA critical points are contained in
a compact deterministic subset in
the control space $L^2(0,1;\mathbb{R}^m)$.
The technical results presented in this section form the basis
for establishing our mean convergence rates in
\cref{sec:ratesvalues,sec:ratescritialponts}.
We recall that $\chi$ is the true problem's criticality
measure, while $\widehat{\chi}_N$ is that of the SAA problem, as defined in the introduction.
We call $u \in \mathrm{dom}(\psi)$ a critical point
of \eqref{eq:optimalcontrolproblem} if
$\chi(u) = 0$. Moreover,
we refer to $u_N \in \mathrm{dom}(\psi)$ as a critical point of
the SAA problem if $\widehat{\chi}_N(u_N) = 0$.

We define the subset $\mathcal{C}\subset L^2(0,1;\mathbb{R}^m)$ by
\begin{align}
\mathcal{C} \coloneqq
\mathrm{prox}_{(1/\alpha)\psi_\alpha}
(- (1/\alpha)\overline{\mathbb{B}}_{W_1^2(0,1;\mathbb{R}^m)}(0;R)),
\end{align}
where we recall $\psi_\alpha(u) =\psi(u) -  (\alpha/2) \|u\|_{L^2(0,1;\mathbb{R}^m)}^2$.

Our next theorem establishes  the covering numbers of $\mathcal{C}$.
Moreover, it demonstrates that each critical point of the
control problem
\eqref{eq:optimalcontrolproblem} and of
its SAA problem
\eqref{eq:saaproblem} is contained in $\mathcal{C}$.

\begin{theorem}
\label{thm:controlregularity}
If \Cref{assumption:feasibleset-regularizer,assumption:dynamicalsystem,assumption:essentiallyboundedtrajectories}
hold, then the following statements are valid.
\begin{enumerate}[nosep]
\item The diameter
of $\mathcal{C}$ in $L^2(0,1;\mathbb{R}^m)$ is at most $2r_{\psi}$.
\item For all $\nu > 0$, the $\nu$-covering number
of $\mathcal{C}$
with respect to the $L^2(0,1;\mathbb{R}^m)$-norm satisfies
\begin{align*}
\log_2 (\mathcal{N}(\mathcal{C},\nu))
\leq \varrho\sqrt{m}(Rm/(\alpha\nu)),
\end{align*}
where $\varrho >0$ is a
constant independent of $m$, $\nu$, $R$, and $\alpha$.
\item We have
\begin{align*}
\int_{0}^{r_{\psi}}
\sqrt{\ln(2\mathcal{N}(\mathcal{C},\varepsilon))}
\, \mathrm{d}\varepsilon
\leq
2\sqrt{r_{\psi}\varrho(Rm^{3/2}/\alpha)} + r_\psi.
\end{align*}
\item Each critical point of
\eqref{eq:optimalcontrolproblem} and of
\eqref{eq:saaproblem} is contained in $\mathcal{C}$.
\end{enumerate}
\end{theorem}

\begin{proof}
\begin{enumerate}[nosep,wide]
\item We have $\mathcal{C} \subset \mathrm{dom}(\psi_\alpha)
=\mathrm{dom}(\psi)$.
\Cref{assumption:feasibleset-regularizer} ensures
that $\mathrm{dom}(\psi) \subset
 \overline{\mathbb{B}}_{L^\infty(0,1;\mathbb{R}^m)}(0;r_\psi)$.
 Hence the diameter
 of $\mathcal{C}$ in $L^2(0,1;\mathbb{R}^m)$ is at most $2r_{\psi}$.
 \item
We define the constant $\varepsilon \coloneqq (\alpha/R) \nu$.
\Cref{prop:coveringnumbers} ensures the
existence of points
$v_1, \ldots, v_K \in \overline{\mathbb{B}}_{W_1^2(0,1;\mathbb{R}^m)}(0;1)$
such that $\|v - v_{k(v)}\|_{L^2(0,1;\mathbb{R}^m)}
\leq \varepsilon$ for all $v \in \overline{\mathbb{B}}_{W_1^2(0,1;\mathbb{R}^m)}(0;1)$, where
$k(v) \coloneqq \mathrm{argmin}_{k \in \{1, \ldots, K\}} \|v - v_{k}\|_{L^2(0,1;\mathbb{R}^m)}$.
Additionally,
$K \leq 2^{\varrho\sqrt{m}(m/\varepsilon)}$,
where
$\varrho > 0$ independent of $m$ and $\varepsilon$.

Let $w\in  \overline{\mathbb{B}}_{W_1^2(0,1;\mathbb{R}^m)}(0;R)$.
Then $v \coloneqq w/R \in \overline{\mathbb{B}}_{W_1^2(0,1;\mathbb{R}^m)}(0;1)$.
Since proximity operators are nonexpansive,
\begin{align*}
&\|\mathrm{prox}_{(1/\alpha)\psi_\alpha}
(- (1/\alpha)w)-
\mathrm{prox}_{(1/\alpha)\psi_\alpha}
(- (1/\alpha)Rv_{k(v)})\|_{L^2(0,1;\mathbb{R}^m)}
\\
& \quad \leq
(1/\alpha)
\|w-Rv_{k(v)}\|_{L^2(0,1;\mathbb{R}^m)}
=
(R/\alpha)
\|v-v_{k(v)}\|_{L^2(0,1;\mathbb{R}^m)}
\leq (R/\alpha) \varepsilon  = \nu.
\end{align*}

 \item
 Combining the previous part and
  \Cref{example:entropyintegralbound}
  (applied with $D \leq 2r_{\psi}$
  and $c=\varrho\sqrt{m}(Rm/\alpha)$),
  we obtain the integral
  bound.
  \item
  Let $\bar u$ be a critical point of \eqref{eq:optimalcontrolproblem}.
  \Cref{lem:optimalityconditions-compositeproblem,lem:ReducedObjectiveLipschitzGateaux}
  ensure
  $$\bar u = \mathrm{prox}_{(1/\alpha)\psi_\alpha}(-(1/\alpha)\nabla_{u} \mathbb{E}[T(\bar u,\xi)]).$$
  Moreover,
  \Cref{thm:gradientregularity} and the dominated convergence theorem
  yield $\nabla_u \mathbb{E}[T(\bar u,\xi)]
  = \mathbb{E}[\nabla_u T(\bar u,\xi)]
  \in L^2(0,1;\mathbb{R}^m)$.
\Cref{thm:gradientregularity}
also ensures
$\nabla_u T(u,\xi) \in W_1^2(0,1;\mathbb{R}^m)$
for all $(u,\xi) \in \mathrm{dom}(\psi) \times \Xi$.
Since $W_1^2(0,1;\mathbb{R}^m)$ is reflexive,
$\overline{\mathbb{B}}_{W_1^2(0,1;\mathbb{R}^m)}(0;R)$ is closed in
$L^2(0,1;\mathbb{R}^m)$.
Combined with Proposition~1.2.12
in \cite{Hytoenen2016},
we find that
$\mathbb{E}[\nabla_u T(\bar u,\xi)]
  \in W_1^2(0,1;\mathbb{R}^m)$.
Hence $\bar u \in \mathcal{C}$.

Now, let $\bar u_N$ be a critical point of the SAA problem
\eqref{eq:saaproblem}. \Cref{lem:optimalityconditions-compositeproblem} yields
$\bar u_N = \mathrm{prox}_{(1/\alpha)\psi_\alpha}(-(1/\alpha)\nabla \widehat{g}_N(\bar u_N))$.
Since $\nabla_u T(u,\xi) \in W_1^2(0,1;\mathbb{R}^m)$
for all $(u,\xi) \in \mathrm{dom}(\psi) \times \Xi$,
we have $\nabla \widehat{g}_N(\bar u_N) = (1/N) \sum_{i=1}^N \nabla_u T(\bar u_N,\xi^i) \in W_1^2(0,1;\mathbb{R}^m)$.
Hence $\bar u_N \in \mathcal{C}$.
\end{enumerate}
\end{proof}

\section{Convergence rates of optimal values}
\label{sec:ratesvalues}

We establish mean convergence rates
for SAA optimal values, our first main contribution.
Throughout the section we assume that the control problems
\eqref{eq:optimalcontrolproblem} and \eqref{eq:saaproblem}
have solutions.
The existence of solutions can be established using standard arguments, once it
is shown that the parameterized objective function is weakly lower semicontinuous
in its first argument; that is, $T(\cdot, \xi)$ is weakly lower semicontinuous
on $\mathrm{dom}(\psi)$
for each $\xi \in \Xi$. This property follows by showing that the mapping
$\mathrm{dom}(\psi) \ni u \mapsto x^u(1,\xi)
\in \mathbb{R}^n$ is weakly continuous. Lemma~2.4 in
\cite{Scagliotti2023} proves  a weak-to-strong continuity result
of solution operators
for control-affine initial value problems. Its proof may be adapted to establish
the weak continuity required here. Now,
let $u^*$ be a solution to \eqref{eq:optimalcontrolproblem}
and for each $N \in \mathbb{N}$,
let $u_N^* \colon \Omega \to L^2(0,1;\mathbb{R}^m)$ be  measurable such that for each $\omega \in \Omega$,
$u_N^*(\omega)$ is a solution to \eqref{eq:saaproblem}.
Moreover, we denote by $v^*$  the optimal value of
\eqref{eq:optimalcontrolproblem}
and by $\widehat{v}_N^* \colon\Omega \to \mathbb{R}$ that of \eqref{eq:saaproblem}, which we also assume to be measurable.
Along with the existence of solutions, the measurability of the SAA
optimal values and solutions can be established using standard arguments.

\begin{theorem}
\label{thm:convergence-optimal-values}
If \Cref{assumption:feasibleset-regularizer,assumption:dynamicalsystem,assumption:essentiallyboundedtrajectories}
hold and $u_0 \in \mathcal{C}$,
then for all $N \in \mathbb{N}$,
\begin{align*}
\mathbb{E}[|\widehat{v}_N^*-v^*|] &\leq
\frac{(\mathbb{E}[(T(u_0,\xi)-\mathbb{E}[T(u_0,\xi)])^2])^{1/2}}
{\sqrt{N}}
+  \frac{16 \sqrt{3} L_T^\prime}{\sqrt{N}}
\Bigg[
\bigg(\frac{r_{\psi}\varrho R m^{3/2}}{\alpha}\bigg)^{1/2}
+ r_{\psi}\Bigg].
\end{align*}
\end{theorem}
\begin{proof}
Since each solution to \eqref{eq:optimalcontrolproblem}
is a critical point of  \eqref{eq:optimalcontrolproblem},
\Cref{thm:controlregularity} ensures  that the solution set
of \eqref{eq:optimalcontrolproblem} is contained
in $\mathcal{C}$. In particular, $u^* \in \mathcal{C}$.
Analogously, each SAA solution is contained in
$\mathcal{C}$. Hence,
\begin{align*}
|\widehat{v}_N^*-v^*|\leq \sup_{u \in \mathcal{C}}
\bigg|\mathbb{E}[T(u,\xi)]-
\frac{1}{N}\sum_{i=1}^N T(u,\xi^i) \bigg|.
\end{align*}

Next, we estimate the supremum's expected value using \Cref{prop:uniformexponentialtailboundaverage}.
For each $\xi \in \Xi$,
\Cref{lem:ReducedObjectiveLipschitzGateaux} ensures that
 $T(\cdot,\xi)$ is
Lipschitz continuous with Lipschitz constant $L_T^\prime$
on $\mathrm{dom}(\psi)$.
We define
$G(u, \xi)  \coloneqq T(u,\xi) - \mathbb{E}[T(u,\xi)]$.
For each $\xi \in \Xi$, $G(\cdot, \xi)$
is Lipschitz continuous with Lipschitz constant $2L_T^\prime$
on $\mathrm{dom}(\psi)$.
Hence $G$ satisfies \eqref{eq:mappingsubgaussian} with
$M = 2 L_T^\prime$.
Now \Cref{prop:uniformexponentialtailboundaverage,%
rem:uniformexponentialtailboundaverage} ensure
\begin{align*}
\mathbb{E}[|\widehat{v}_N^*-v^*|] &\leq
\frac{(\mathbb{E}[(T(u_0,\xi)-\mathbb{E}[T(u_0,\xi)])^2])^{1/2}}
{\sqrt{N}}
+  \frac{8 \sqrt{3} L_T^\prime}{\sqrt{N}} \int_{0}^{r_{\psi}}
\sqrt{\ln(2\mathcal{N}(\mathcal{C},  \varepsilon))}
\, \mathrm{d}\varepsilon.
\end{align*}
Combined with the
entropy integral bound provided by \Cref{thm:controlregularity}, we obtain
the assertion.
\end{proof}

\section{Convergence rates for criticality measures}
\label{sec:ratescritialponts}

We establish mean convergence rates
for SAA critical points, our second main contribution. We measure accuracy of SAA critical points
via the criticality measure defined in \eqref{eq:truecriticalitymeausre}.

\begin{theorem}
Let \Cref{assumption:feasibleset-regularizer,assumption:dynamicalsystem,assumption:essentiallyboundedtrajectories} hold.
For each $N \in \mathbb{N}$,
let the control $u_N^*\colon \Omega \to L^2(0,1;\mathbb{R}^m)$ be  measurable such that  for each $\omega \in \Omega$,
$u_N^*(\omega) \in \mathrm{dom}(\psi)$ is a critical point
of the SAA problem \eqref{eq:saaproblem}.
If  $u_0 \in \mathcal{C}$,
then for all $N \in \mathbb{N}$,
\begin{align*}
\mathbb{E}[\chi(u_N^*)] &\leq
\frac{(\mathbb{E}\|\nabla_u T(u_0,\xi)-\mathbb{E}[\nabla_uT(u_0,\xi)]\|_{L^2(0,1;\mathbb{R}^m)}^2)^{1/2}}
{\sqrt{N}}
\\
& \quad +  \frac{16 \sqrt{3} L_{\nabla T}^\prime}{\sqrt{N}}
\Bigg[
\bigg(\frac{r_{\psi}\varrho R m^{3/2}}{\alpha}\bigg)^{1/2}
+ r_{\psi}\Bigg].
\end{align*}
\end{theorem}

\begin{proof}
Our proof follows steps similar to those in the proof of
\Cref{thm:convergence-optimal-values}.
For all $u \in \mathrm{dom}(\psi)$,
the nonexpansiveness of
proximity operators ensures
\begin{align*}
|\chi(u) - \widehat{\chi}_N(u)|
\leq
\|\nabla g(u) - \nabla \widehat{g}_N(u)\|_{L^2(0,1;\mathbb{R}^m)}.
\end{align*}
Combined with $\widehat{\chi}_N(u_N^*)  = 0$ and
$u_N^* \in \mathcal{C}$, we obtain
\begin{align*}
\mathbb{E}[\chi(u_N^*)]
\leq
\mathbb{E}\big[\sup_{u \in \mathcal{C}}\,
\|\nabla g(u) - \nabla \widehat{g}_N(u)\|_{L^2(0,1;\mathbb{R}^m)}\big].
\end{align*}

Next, we estimate the above right-hand side
 using \Cref{prop:uniformexponentialtailboundaverage}.
 For each $\xi \in \Xi$,
 \Cref{lem:ReducedObjectiveLipschitzGateaux} ensures that
 $\nabla_u T(\cdot,\xi)$ is
 Lipschitz continuous with Lipschitz constant $L_{\nabla T}^\prime $
 on $\mathrm{dom}(\psi)$.
 We define
 $G(u, \xi)  \coloneqq \nabla_u T(u,\xi) - \mathbb{E}[\nabla_u T(u,\xi)]$.
 For each $\xi \in \Xi$, $G(\cdot, \xi)$
 is Lipschitz continuous with Lipschitz constant $2L_{\nabla T}^\prime $
 on $\mathrm{dom}(\psi)$.
 Hence $G$ satisfies \eqref{eq:mappingsubgaussian} with
 $M = 2 L_{\nabla T}^\prime $.
 Combined with \Cref{thm:controlregularity,prop:uniformexponentialtailboundaverage,%
prop:uniformexponentialtailboundaverage,%
 rem:uniformexponentialtailboundaverage},
 we obtain the assertion.
\end{proof}

\section{Numerical illustrations}
\label{sec:numericalillustrations}

We consider two numerical examples. The main purpose of our numerical
illustrations is to demonstrate the typical Monte Carlo convergence
rate for the SAA optimal values and SAA critical points established
in \cref{sec:ratesvalues,sec:ratescritialponts}.
Before discussing these examples,
we provide details on the discretization of the control problems,
implementation details, and a description of our simulation output.
We also introduce some terminology.

We refer to
\begin{align*}
\min_{u \in L^2(0,1;\mathbb{R}^m)} \,
F(x^u(1,\mathbb{E}[\xi]), \mathbb{E}[\xi]) + \psi(u),
\end{align*}
as the nominal problem to \eqref{eq:optimalcontrolproblem}.
It is obtained by replacing the random element $\xi$ by its expected value
and simulating the dynamical system \eqref{eq:uncertain-ode}
with this expected value. For both numerical examples, we visualize the nominal solution and the solution to a certain SAA problem, which is introduced below
and referred to as the reference problem.
These graphical illustrations demonstrate that nominal solutions and reference solutions can have different characteristics.

To approximate the optimal value of \eqref{eq:optimalcontrolproblem}
and the criticality measures $\chi$,
we generated $N_{\text{ref}} \coloneqq 4096$ samples
using a realization of a scrambled Sobol'
sequence and transformed these samples to take values in
$\Xi$. For both examples, $\Xi$ is the Cartesian product of finitely many closed intervals and the components of $\xi$ are uniformly distributed. We denote the samples by $\zeta^{1}, \ldots, \zeta^{N_{\text{ref}}}$,
and formulate the reference problem
\begin{align*}
\min_{u \in L^2(0,1;\mathbb{R}^m)} \,
\frac{1}{N_{\text{ref}}} \sum_{i=1}^{N_{\text{ref}}}
F(x^u(1,\zeta^i), \zeta^i) + \psi(u).
\end{align*}
Moreover, we approximate the criticality measure $\chi$ by
\begin{align*}
\chi_{\text{ref}}(u) \coloneqq
\Big\|u-
\mathrm{prox}_{\psi_\alpha}\Big(u-
\frac{1}{N_{\text{ref}}} \sum_{i=1}^{N_\text{ref}} \nabla_u T(u,\zeta^{i})-\alpha  u\Big)
\Big\|_{L^2(0,1;\mathbb{R}^m)}.
\end{align*}

In addition to providing plots of the nominal and reference solutions, we
compare their quality using performance metrics from~\cite{Birge1982}.
We define the (reference)
expected performance of the nominal policy (EPP) as
the reference SAA objective evaluated at the nominal solution, and the reference
performance (RP) as the optimal value of the reference SAA problem. The
reference  value of
the stochastic solution (VSS) is defined as $\text{EPP} - \text{RP}$, and the
relative VSS as $\text{VSS} / \text{EPP}$.

The dynamical systems are discretized using
multiple shooting with a fixed step explicit fourth-order Runge--Kutta method.
The control space is discretized using piecewise constant functions
defined on a uniform grid over $[0,t_f]$ with $q \coloneqq 50$ subintervals,
where $t_f > 0$ is a fixed final time.
Consequently, the solutions to the SAA problems
and the criticality measures depend on the discretization parameter $q$. For notational convenience, this discretization parameter is omitted in the notation.
We generated Sobol' sequences using \texttt{SciPy}'s
quasi-Monte Carlo capabilities \cite{Roy2023}.
To model and solve the discretized SAA problems, we used the package
\texttt{EnsembleControl} \cite{Milz2024},
which is built on the software framework
\texttt{CasADi} \cite{Andersson2019}, and
the optimization solver \texttt{IPOPT} \cite{Waechter2006}.
We used \texttt{IPOPT} with its default settings.
Our simulation code and output are
archived at \cite{Milz2024c}. Simulations were performed on a laptop equipped with a 12th Gen Intel(R) Core(TM) i7-1260P processor and 16 GB of RAM\@.

\subsection{Harmonic oscillator}
\label{sec:harmonicoscillator}

\begin{figure}[t]
	\centering
	\includegraphics[width=0.495\textwidth]{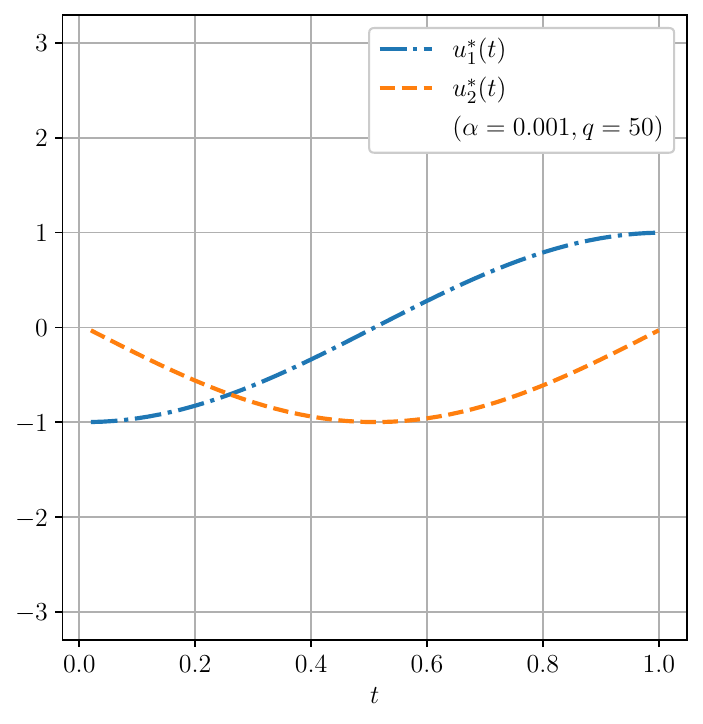}
	\subfloat{%
		\includegraphics[width=0.495\textwidth]{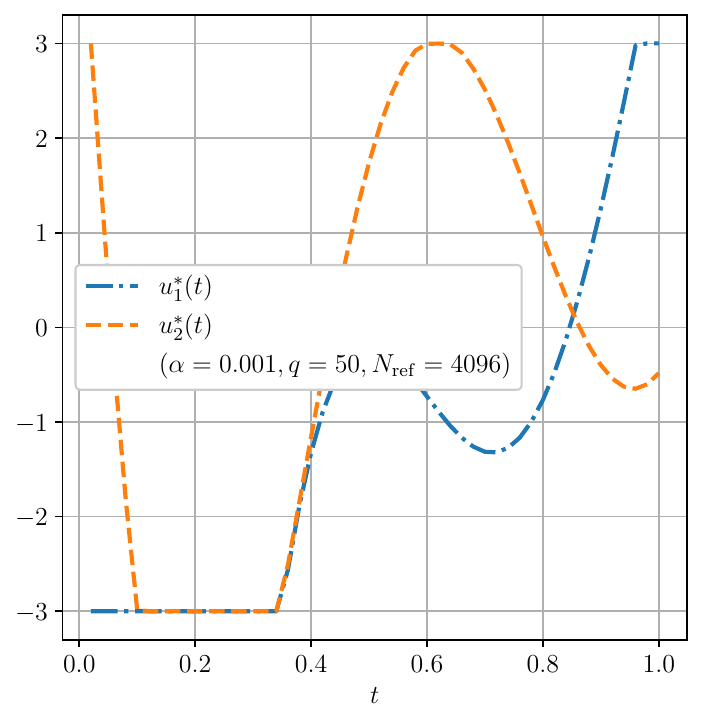}}
	\caption{%
		For the harmonic oscillator problem formulated in
		\cref{sec:harmonicoscillator},
		nominal solution \emph{(left)} and
		reference solution \emph{(right)}.}
	\label{fig:harmonic_oscillator_solutions}
\end{figure}

\begin{figure}[t]
	\centering
	\subfloat{%
		\includegraphics[width=0.495\textwidth]{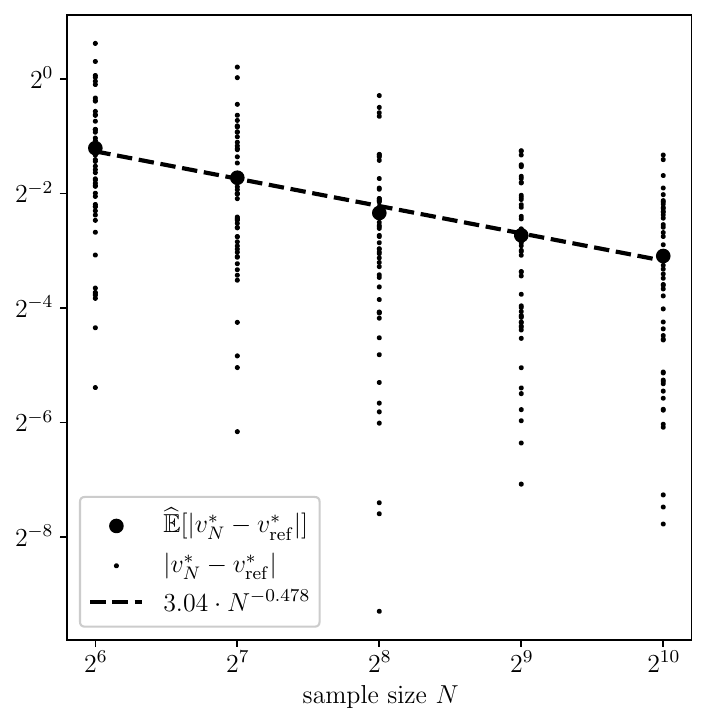}}
	\subfloat{%
		\includegraphics[width=0.495\textwidth]{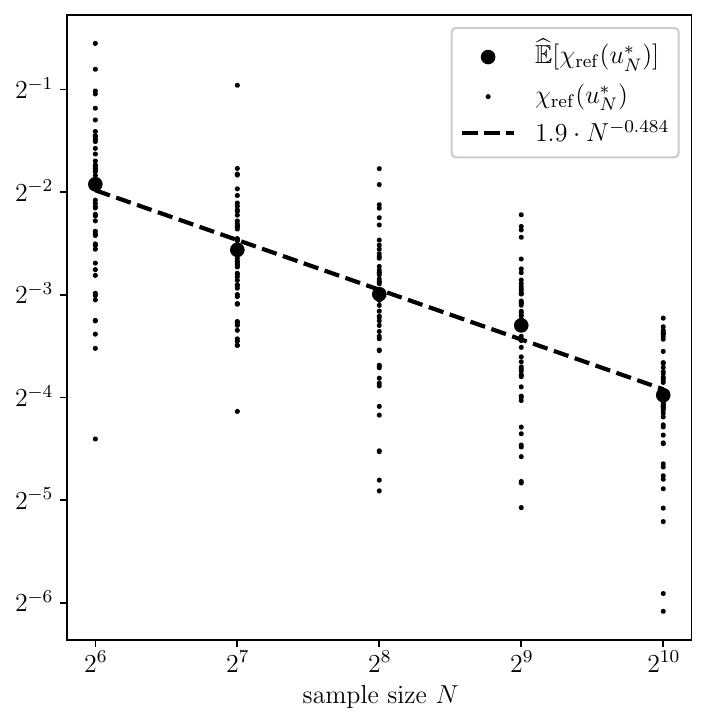}}
	\caption{%
		For the harmonic oscillator problem formulated in
		\cref{sec:harmonicoscillator},  convergence rate of the
		SAA optimal values  \emph{(left)}  and of the reference criticality measure
		evaluated at SAA critical points \emph{(right)}.
		The empirical means $\widehat{\mathbb{E}}$ were computed using $50$ replications. The convergence rates were computed using least squares.}
	\label{fig:harmonic_oscillator_rates}
\end{figure}

We consider a harmonic oscillator
inspired by the problem
formulation in \cite[sect.\ 7.1]{Phelps2016}.
Our optimization problem is given by
\begin{align*}
\min_{u \in L^2(0,1;\mathbb{R}^2)}\,
(1/2)\mathbb{E}[x_1^u(1,\xi)^2 + x_2^u(1,\xi)^2] + \psi(u),
\end{align*}
where for each parameter $\xi \in \Xi \subset \mathbb{R}^5$
and  control $u(\cdot) \in L^2(0,1;\mathbb{R}^2)$,
$x^u(\cdot,\xi) = x(\cdot,\xi)$ solves
\begin{align*}
\dot{x}(t,\xi) =
\begin{bmatrix}
0 & - \xi_1 \\
\xi_1 & 0
\end{bmatrix}
x(t,\xi) +
\begin{bmatrix}
u_1(t) \\ u_2(t)
\end{bmatrix}
+
\begin{bmatrix}
	\xi_2 \\ \xi_3
\end{bmatrix}
, \quad t \in (0,1), \quad x(0,\xi) =
\begin{bmatrix}
1+\xi_4\\ \xi_5
\end{bmatrix}.
\end{align*}
Moreover, the function $\psi(\cdot)$ is the sum of
$(\alpha/2)\|\cdot\|_{L^2(0,1;\mathbb{R}^2)}^2$
with $\alpha = 10^{-3}$,
and the indicator function of
the set $\{u \in L^2(0,1;\mathbb{R}^2) \colon
\|u_i\|_{L^\infty(0,1;\mathbb{R})}\leq 3, \quad i = 1, 2
\}$.
The random variables
$\xi_1, \ldots, \xi_5$
are independent, and uniformly distributed on
$[0, 2\pi]$, $[-5,5]$, $[-5,5]$, $[-1/2,1/2]$,
and $[-1/2,1/2]$, respectively.
The Cartesian product of these intervals
defines the sample space $\Xi$.

\Cref{fig:harmonic_oscillator_solutions} depicts the nominal
and reference solutions. \Cref{fig:harmonic_oscillator_rates} depicts
convergence rates for the SAA optimal values and critical points.
These empirical rates are close to the theoretical rates established in
\cref{sec:ratesvalues,sec:ratescritialponts}.
We have $\text{EPP} = 4.14307$, $\text{RP} = 3.85392$, yielding
$\text{VSS} = 0.28915$ and hence a relative VSS of  $6.98\%$.

\subsection{Vaccination scheduling for epidemic control under model parameter uncertainty}
\label{subsec:vaccination}

\begin{figure}[t]
	\centering
	\includegraphics[width=0.495\textwidth]{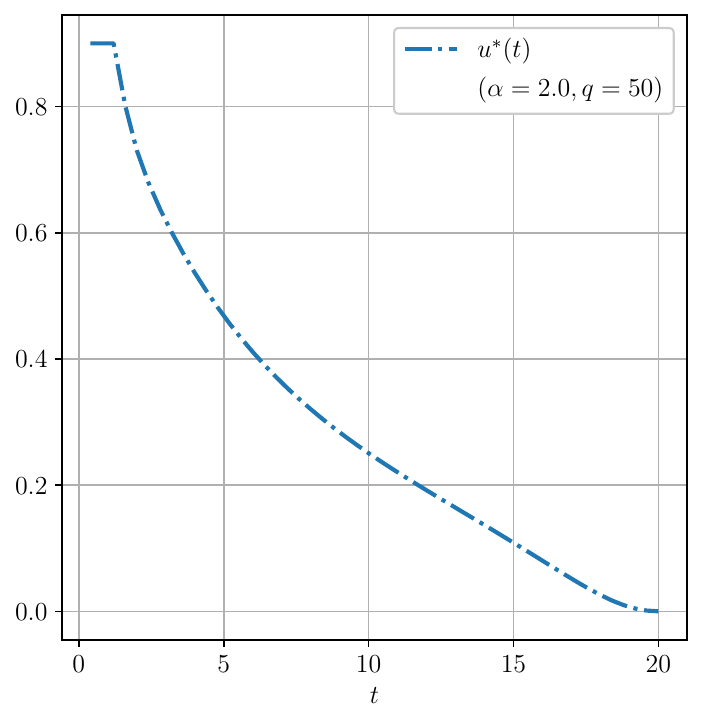}
	\subfloat{%
		\includegraphics[width=0.495\textwidth]{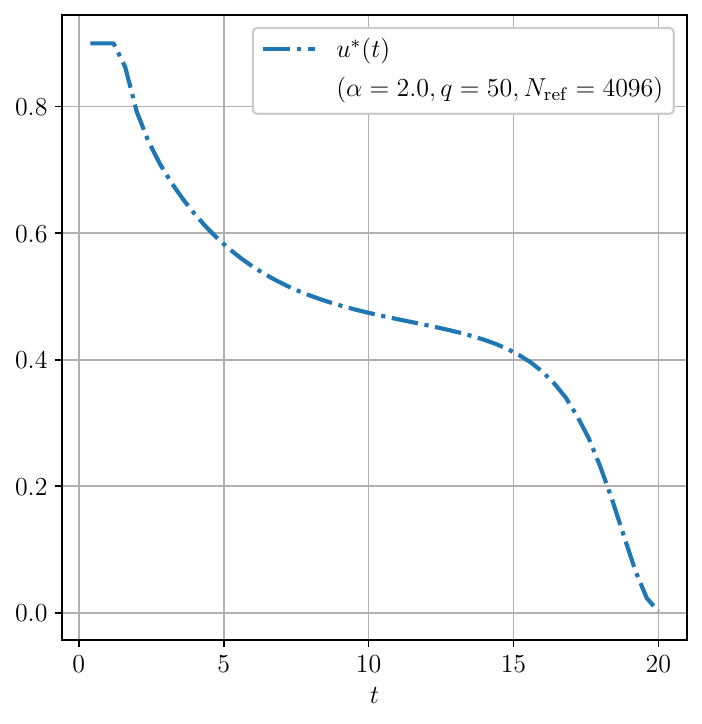}}
	\caption{%
		For the vaccination scheduling problem formulated in
		\cref{subsec:vaccination},
		nominal solution \emph{(left)} and
		reference solution \emph{(right)}.}
	\label{fig:vaccination_solutions}
\end{figure}

We formulate an optimal control problem to obtain vaccination schedules
for an epidemic disease under model parameter uncertainty. This problem
formulation is inspired by the optimal control problem formulated in
\cite[sect.\ 4]{MillerNeilan2010}
(see also \cite[Chap.\ 13]{Lenhart2007}).
The dynamics are based on an SEIR model described in
\cite[sect.\ 4.2]{MillerNeilan2010}, which tracks the number of individuals
who are susceptible (S), exposed (E), infected (I), and recovered/vaccinated (R), as well as the total population.
The model relies on six key parameters: natural death
rate, disease death rate, birth rate, infection incidence, infection rate, and recovery rate.
We account for uncertainty by treating these six parameters as random variables.
Our control variable is the time-dependent vaccination rate for susceptible individuals.
The goal is to minimize the expected number of infected individuals and the cost of vaccination.

We formulate the optimal control problem
\begin{align}
	\label{eq:vaccination}
	\min_{u \in L^2(0, t_f;\mathbb{R})}\,
	(1/10)\mathbb{E}\Big[\int_{0}^{t_f} I^u(t,\xi) \mathrm{d} t\Big] + \psi(u),
\end{align}
where for each  control $u(\cdot) \in L^2(0, t_f;\mathbb{R})$
with $0 \leq u(t) \leq 0.9$ for a.e.\ $t \in (0,t_f)$ and parameter
$\xi \coloneqq (a, b, c, d, e, g) \in \Xi$,
the states
$S^u(\cdot,\xi)$, $E^u(\cdot, \xi)$, $I^u(\cdot,\xi)$, $R^u(\cdot,\xi)$,
$M^u(\cdot, \xi)$
solve the SEIR model (see \cite[eqns.\ (4.1)--(4.5)]{MillerNeilan2010})
\begin{align*}
	\dot{S}(t) & =b M(t)-d S(t)-c S(t) I(t)-u(t) S(t),
	& S(0)  &= S_0, \\
	\dot{E}(t) & =c S(t) I(t)-(e+d) E(t),  &E(0)  &= E_0, \\
	\dot{I}(t) & =e E(t)-(g+a+d) I(t), & I(0)  &= I_0, \\
	\dot{R}(t) & =g I(t)-d R(t)+u(t) S(t), & R(0)  &= R_0,  \\
	\dot{M}(t) & =(b-d) M(t)-a I(t), & M(0)  &= M_0.
\end{align*}

Following \cite[p.\ 76]{MillerNeilan2010}, we choose $\psi(\cdot)$
as the sum of the indicator function of
$\{u \in L^2(0, t_f; \mathbb{R})
\colon 0 \leq u(t) \leq 0.9 \text{ for a.e. } t \in (0,t_f)\}$ and
$(\alpha/2)\|\cdot\|_{L^2(0, t_f; \mathbb{R})}^2$
with $\alpha = 2$.
For the numerical solution, we removed the fourth state equation,
as none of the other states depend on $R(\cdot)$.
The control problem \eqref{eq:vaccination} can be formulated as an instance
of the risk-neutral problem \eqref{eq:optimalcontrolproblem} using a time transformation and incorporating an additional state variable.

We describe the parameter values and our
choice of the random vector $\xi \in \mathbb{R}^6$
used for our numerical simulations.
As in \cite[Table 1]{MillerNeilan2010}, we use the initial states
$S_0 = 1000$,
$E_0 = 100$,
$I_0 = 50$, and
$R_0 = 15$.
Let $M_0 = S_0 + E_0 + I_0 + R_0$,
and let $t_f = 20$.
We choose the nominal parameter
$\bar \xi = (0.2, 0.525, 0.001, 0.5, 0.5, 0.1)$
for $\xi = (a, b, c, d, e, g)$
using the parameter values in \cite[Table 1]{MillerNeilan2010}.
We construct the random variables $\xi_j$, $j = 1, \ldots, 6$, through
random relative
perturbations of the nominal parameter $\bar \xi$. Specifically,
for $\sigma \coloneqq 0.15$ and independent random variables
$\rho_j$, $j = 1, \ldots, 6$, each uniformly distributed on
the interval $[-1,1]$, we define
\begin{align*}
	\xi_j \coloneqq (1+\sigma \rho_j)\bar \xi_j,
	\quad j = 1, \ldots, 6.
\end{align*}
Therefore, $\Xi$ is chosen as the Cartesian product
of the intervals $[(1-\sigma)\bar \xi_j, (1+\sigma)\bar \xi_j]$, $j=1, \ldots, 6$.

\Cref{fig:vaccination_solutions} depicts the nominal
and reference solutions. \Cref{fig:vaccination_rates} depicts
convergence rates for the SAA optimal values and critical points.
These empirical rates closely match the theoretical rates established in
\cref{sec:ratesvalues,sec:ratescritialponts}. Given the potential nonconvexity of the vaccination scheduling problem, we cannot guarantee that the optimal values shown in \Cref{fig:vaccination_rates}, as computed by the optimization solver, correspond to the true optimal values.
We have $\text{EPP} = 31.5168$,
$\text{RP} = 26.6276$, $\text{VSS} = 4.88922$, and
hence a relative VSS of $15.51\%$.

The optimal control problem \eqref{eq:vaccination} may not satisfy
\Cref{assumption:essentiallyboundedtrajectories}. 
Using the general invariance principle stated in  \cite[Thm.\ on p.~119]{Walter1998}
(see also the discussions in \cite{Hethcote2000}), we can determine an invariant set for the SEIR dynamics that is 
independent of the feasible controls and parameters in $\Xi$.
This result can be combined with the
construction proposed in \cite{Scagliotti2025} to obtain modified dynamics and a corresponding optimal control problem
that satisfy \Cref{assumption:essentiallyboundedtrajectories}; see
\cite[eqns.\ (2.6) and (2.19), and Lemma~2.3]{Scagliotti2025}
for further details.

\begin{figure}[t]
	\centering
	\includegraphics[width=0.495\textwidth]{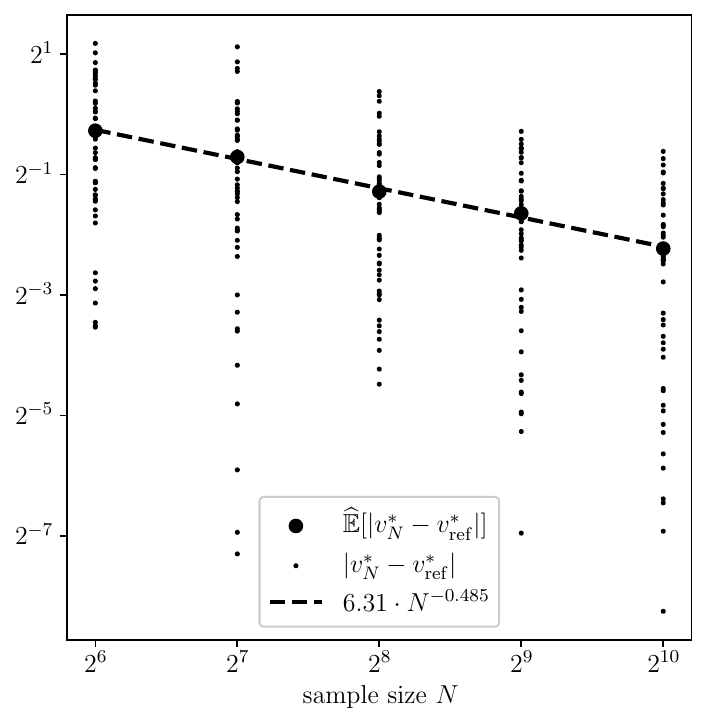}
	\subfloat{%
		\includegraphics[width=0.495\textwidth]{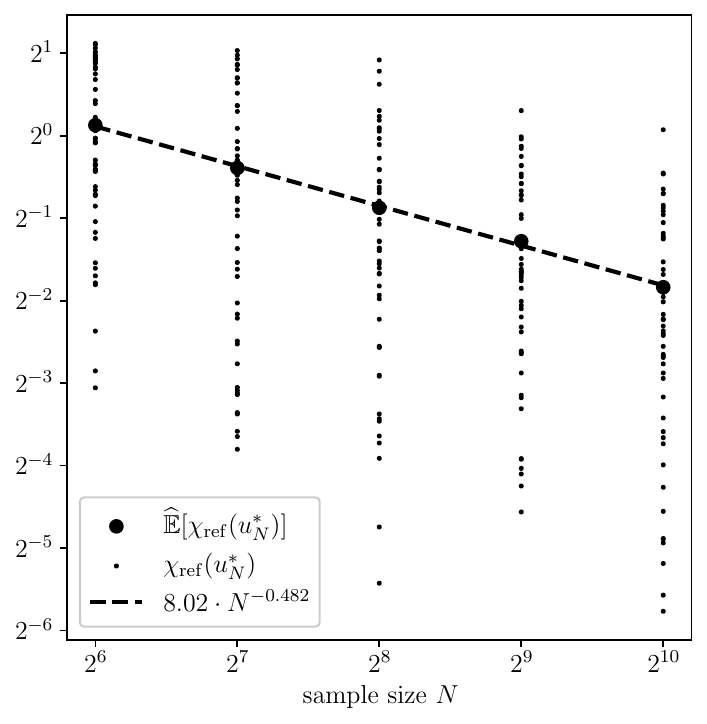}}
	\caption{%
		For the vaccination scheduling problem formulated in
		\cref{subsec:vaccination},  convergence rate of the
		SAA optimal values as computed by the optimization solver  \emph{(left)}  and of the reference criticality measure
		evaluated at SAA critical points \emph{(right)}.
		The empirical means $\widehat{\mathbb{E}}$ were computed using $50$ replications. The convergence rates were computed using least squares.}
	\label{fig:vaccination_rates}
\end{figure}

\section{Discussion}

We have considered risk-neutral optimal control of affine-control dynamics with random inputs. This manuscript has demonstrated nonasymptotic convergence rates for SAA optimal values and SAA critical points, which align with the typical Monte Carlo convergence rate. We have empirically validated these rates through numerical simulations.
An open question remains whether the dependence on the strong convexity parameter in the convergence rates is optimal. This consideration may become particularly relevant for applications of our results to (statistical) inverse problems.

\subsection*{Acknowledgments}
The second author thanks Ashwin Pananjady for insightful discussions on
entropy integral bounds, optimality, and regularization.
The authors thank the two anonymous referees for their comments and suggestions, which have improved the quality of this manuscript.

\subsection*{Reproducibility of computational results}
Computer code that allows the reader to reproduce the computational results in
this manuscript is available at
\url{https://doi.org/10.5281/zenodo.15529367}.

\appendix
\section{Optimality conditions}
\label{sec:optimality-conditions}
We state essentially known first-order necessary optimality
conditions for composite optimization in Hilbert spaces. We use these
optimality conditions to derive first-order optimality conditions
for the control problem under uncertainty and its SAA problem,
and to define criticality measures. These optimality conditions
are used to justify the definition of certain criticality measures.

We consider the composite optimization problem
\begin{align}
\label{eq:compositeproblem}
\min_{u \in \mathrm{dom}(\psi)} \, G(u) + \psi(u).
\end{align}

We formulate assumptions on the optimization problem
\eqref{eq:compositeproblem}.

\begin{assumption}
\label{assumption:compositeproblem}
  \begin{enumerate}[nosep]
  \item The space $U$ is a real Hilbert space.
  \item  The function $\psi \colon U \to (-\infty,\infty]$ is
  proper, lower semicontinuous, and strongly convex
  with parameter $\alpha > 0$.
  \item The set  $U_1 \subset U$
  is convex with $\mathrm{dom}(\psi) \subset U_1$. The function $G\colon U_1 \to \mathbb{R}$
  is Gateaux differentiable on $\mathrm{dom}(\psi)$
  relative to $U_1$.
  \end{enumerate}
\end{assumption}

Let \Cref{assumption:compositeproblem} hold true.
We say that $\bar u \in \mathrm{dom}(\psi)$ is a critical point of
\eqref{eq:compositeproblem} if
$- \nabla G(\bar u) \in \partial \psi(\bar u)$.
Here, $\partial \phi(u)$ denotes the subdifferential
of $\phi \colon U \to (-\infty,\infty]$ at $u \in U$.
Let us define $\psi_\alpha(u) \coloneqq \psi(u) -  (\alpha/2) \|u\|_U^2$.

The following essentially known fact characterizes the
critical points of \eqref{eq:compositeproblem}.

\begin{lemma}
\label{lem:optimalityconditions-compositeproblem}
Let \Cref{assumption:compositeproblem} be satisfied.
If $\bar u \in \mathrm{dom}(\psi)$ is a local solution to \eqref{eq:compositeproblem},
then $- \nabla G(\bar u) \in \partial \psi(\bar u)$.
If $\bar u \in \mathrm{dom}(\psi)$, then
$\bar u $ is a critical point of \eqref{eq:compositeproblem}
if and only if
$$
\bar u
= \mathrm{prox}_{\gamma\psi_\alpha}(\bar u-\gamma\nabla G(\bar u)-\alpha \gamma \bar u)
\quad \text{for any and all} \quad \gamma > 0.
$$
\end{lemma}
\begin{proof}
The necessary first-order optimality condition is a direct consequence
of well-known statements.
Let us define $G_\alpha(u) \coloneqq G(u) + (\alpha/2) \|u\|_U^2$.
We have
$\partial \psi(u)  =\partial \psi_\alpha(u) + \alpha u$.
Therefore, $\bar u$ is a critical point for \eqref{eq:compositeproblem}
if and only if
$- \nabla G_\alpha(\bar u) \in \partial \psi_\alpha(\bar u)$
if and only if
$- \gamma \nabla G_\alpha(\bar u) \in \partial [\gamma\psi_\alpha(\bar u)]$.
Combined with
\cite[Prop.\ 12.26]{Bauschke2011}, we obtain the
characterizations of critical points of \eqref{eq:compositeproblem}.
\end{proof}

\section{Covering numbers of vector-valued functions}
\label{sec:coveringnumbers}
We provide upper bounds on the covering numbers of
$\overline{\mathbb{B}}_{W_1^{2}(0,1; \mathbb{R}^m)}(0;1)$
with respect to the $L^2(0,1;\mathbb{R}^m)$-norm.
The covering numbers may be nonoptimal
with respect to the space dimension $m$, but are optimal
with respect to the covering radii.
In the main text, we use these covering numbers to establish
those of a deterministic set containing all SAA critical points.

\begin{proposition}
\label{prop:coveringnumbers}
For some $\varrho > 0$
and each $m \in \mathbb{N}$,
the binary logarithm of the $\nu$-covering number of $\overline{\mathbb{B}}_{W_1^{2}(0,1; \mathbb{R}^m)}(0;1)$ with respect to the $L^2(0,1;\mathbb{R}^m)$-norm
does not exceed $\varrho\sqrt{m}(m/\nu)$ for
all sufficiently small $\nu> 0$.
\end{proposition}

\begin{proof}
According to Theorem~1.7 in \cite{Birman1980},
the binary logarithm of the $\nu$-covering number of  $\overline{\mathbb{B}}_{W_1^{2}(0,1)}(0;1)$ with respect to the $L^2(0,1)$-norm
is proportional to $(1/\nu)$ for all sufficiently small $\nu> 0$.
In particular, there exists $\varrho > 0$
such that for each $\nu > 0$,
this $\nu$-covering number does not exceed
$2^{(\varrho/2) (1/\nu)}$.
Fix $\nu > 0$. Let $y_1, \ldots, y_K$ be a $(\nu/\sqrt{m})$-cover of $\overline{\mathbb{B}}_{W_1^{2}(0,1)}(0;1)$ with respect to the $L^2(0,1)$-norm.
We have $K \leq 2^{(\varrho/2) (\sqrt{m}/\nu)}$.
We consider
$\mathcal{W}_K \coloneqq \{w \in L^2(0,1;\mathbb{R}^m) \ \colon w_i \in \{y_1, \ldots, y_K\},
\quad i = 1, \ldots, m\}$.
Fix $v \in \overline{\mathbb{B}}_{W_1^{2}(0,1; \mathbb{R}^m)}(0;1)$.
Hence $v_i \in \overline{\mathbb{B}}_{W_1^{2}(0,1)}(0;1)$, $i = 1, \ldots, m$.
We can choose $w \in \mathcal{W}_K$
such that  for each $i \in \{1, \ldots, m\}$,
$ \|v_i - w_i\|_{L^2(0,1)} \leq \nu/\sqrt{m}$.
Hence
\begin{align*}
\|v- w\|_{L^2(0,1;\mathbb{R}^m)}^2
= \sum_{i=1}^m\|v_i - w_i\|_{L^2(0,1)}^2
\leq \nu^2.
\end{align*}
The number of elements of $\mathcal{W}_K$
does not exceed $K^m  \leq 2^{(\varrho/2)  m(\sqrt{m}/\nu)}$.

Now, let $Y_0 =\overline{\mathbb{B}}_{W_1^{2}(0,1; \mathbb{R}^m)}(0;1)$
and $Y = L^2(0,1;\mathbb{R}^m)$.
From the proof of Lemma~\mbox{8.2-2}~(c) in \cite{Kreyszig1978}, we deduce
$\mathcal{N}(Y_0; \nu) \leq \mathcal{N}(Y, Y_0; \nu/2)$
for all $\nu  > 0$. Combining the pieces ensures the assertion.
\end{proof}

\section{Uniform expectation bounds in Hilbert spaces}
\label{section:uniformexpectation}
We establish a uniform expectation bound
for sample averages
of independent Hilbert space-valued random variables
using a chaining argument
(see, for example, \cite{Gine2016,Talagrand2021})
and an exponential moment bound
\cite{Pinelis1986}.
The uniform expectation bound allows us to establish
convergence rates for SAA optimal values and SAA critical points,
as demonstrated in the main text.

Throughout the section,
let $(\Theta, \mathcal{A}, \mathbb{P})$ be a complete
probability space, and let $\xi$ be a $\Xi$-valued random element,
where $\Xi$ is a complete separable metric space.

\begin{theorem}
\label{prop:uniformexponentialtailboundaverage}
Let $H$ be a real, separable Hilbert space,
let $C \subset H$ be a nonempty, closed set
with diameter $D > 0$,
and let $G \colon C \times \Xi \to H$ be a Carath\'eodory mapping
with
$\mathbb{E}[\|G(x,\xi)\|_H]  < \infty$ and
$\mathbb{E}[G(x,\xi)] = 0$
for each $x \in C$.
Suppose  that there exists a constant $M > 0$ such that
\begin{align}
\label{eq:mappingsubgaussian}
\mathbb{E}[\cosh(\lambda \|G(x,\xi) - G(y,\xi)\|_H)] \leq
\mathrm{e}^{(1/2)M^2 \lambda^2 \|x-y\|_H^2}
\; \text{for all} \;x, y \in C,
\; \lambda \geq 0.
\end{align}
Let $\xi^1, \xi^2, \ldots$ be independent
$\Xi$-valued random elements defined
on $(\Theta, \mathcal{A}, \mathbb{P})$, each having the same
distribution as $\xi$.
We define
$\widehat{G}_N(x)
\coloneqq (1/N)\sum_{i=1}^N
G(x,\xi^i)$.
Then for each $x_0 \in C$ and $N \in \mathbb{N}$,
\begin{align*}
\mathbb{E}[\sup_{x \in C}\,
\|\widehat{G}_N(x)\|_H]
\leq \mathbb{E}[\|\widehat{G}_N(x_0)\|_H]
+  \frac{4 \sqrt{3} M }{\sqrt{N}} \int_{0}^{D/2}
\sqrt{\ln(2\mathcal{N}(C, \varepsilon))}
\, \mathrm{d}\varepsilon.
\end{align*}
\end{theorem}

Under the hypotheses of \Cref{prop:uniformexponentialtailboundaverage},
we have (cf.\ \cite[p.\ 79]{Yurinsky1995})
\begin{align*}
 \mathbb{E}[\|\widehat{G}_N(x_0)\|_H]
\leq
\sqrt{\frac{ \mathbb{E}[\|G(x_0,\xi)\|_H^2]}{N}}.
\end{align*}
We leverage this bound in the main text.

Before we prepare our proof
of \Cref{prop:uniformexponentialtailboundaverage},
we comment on the condition \eqref{eq:mappingsubgaussian}.

\begin{remark}
\label{rem:uniformexponentialtailboundaverage}
\begin{enumerate}[nosep]
\item For a mean-zero random variable $Z \colon \Xi \to \mathbb{R}$,
the condition,
$\mathbb{E}[\cosh(\lambda |Z|)] \leq \exp(\lambda^2 \tau^2/2)$
for all $\lambda \in \mathbb{R}$ and some $\tau > 0$, is
equivalent
to $Z$ being sub-Gaussian (cf.\ \cite[Lem.\ B.2]{Milz2022b}).
This provides one motivation for using the condition
\eqref{eq:mappingsubgaussian} to model sub-Gaussian behavior.
\item
If $G(\cdot,\xi)$ is Lipschitz continuous with Lipschitz constant
$M$ for each $\xi \in \Xi$,
then \eqref{eq:mappingsubgaussian} holds true.
This is a consequence of the inequality
$\cosh(x) \leq \exp(x^2/2)$ valid for all $x \in \mathbb{R}$.
\end{enumerate}
\end{remark}

We prepare our proof of
\Cref{prop:uniformexponentialtailboundaverage}.
The following lemma provides a basic upper bound on the
expected pointwise maximum of sub-Gaussian random variables.

\begin{lemma}[{see \cite[Lem.\ B.5]{Milz2022b}}]
\label{lem:meanmax}
Let $\sigma > 0$.
If $Z_k : \Theta \to \mathbb{R}$
are random variables
with $\mathbb{E}[\cosh(\lambda |Z_k|)] \leq \exp(\lambda^2\sigma^2/2)$
for all $\lambda \in \mathbb{R}$ and
$k = 1, 2, \ldots, K$, then
$
\mathbb{E}[\max_{1\leq k \leq K} |Z_k|]
\leq \sigma \sqrt{2\ln(2K)}
$.
\end{lemma}

The following proposition is inspired by entropy integral bounds
for sub-Gaussian processes (see, for example,
\cite{Buldygin2000,Gine2016}).

\begin{proposition}
\label{thm:chaininghilbertspace}
Let $(Y_1, d_{Y_1})$ be a separable metric space,
let $C \subset Y_1$ be a nonempty, closed set
with diameter $D_1 \in (0,\infty]$,
let $Y_2$ be a separable Banach space,
and let $G \colon C \times \Xi \to Y_2$
be a Carath\'eodory mapping.
Suppose that
\begin{align}
\label{eq:mappingsubgaussian"}
\mathbb{E}[\cosh(\lambda \|G(x,\xi) - G(y,\xi) \|_{Y_2})] \leq
\exp(\lambda^2 d_{Y_1}(x,y)^2/2)
\; \text{for all} \;x, y \in C,
\; \lambda \geq 0.
\end{align}
Then for all  points $x_0\in C$, we have
\begin{align}
\label{eq:dudley-bound}
\mathbb{E}[\sup_{x \in C}\, \|G(x,\xi)\|_{Y_2}]
\leq \mathbb{E}[ \|G(x_0,\xi)\|_{Y_2}]
+ 4 \sqrt{2} \int_{0}^{D_1/2}
\sqrt{\ln(2\mathcal{N}(C, \varepsilon))}
\, \mathrm{d}\varepsilon.
\end{align}
\end{proposition}

\begin{proof}
We adapt the proof of Theorem~3.1 in \cite[p.\ 95]{Buldygin2000}.
If the right-hand side of \eqref{eq:dudley-bound} is
infinite, then there is nothing to show.
Next, let the right-hand side of \eqref{eq:dudley-bound} be finite.
Hence
$C$ is totally bounded and $D_1$ is finite.
We define $\varepsilon_0 \coloneqq D_1$
and $\varepsilon_k \coloneqq (1/2)^k \varepsilon_0$.
For each $k \in \mathbb{N}$, let $C_k$
be a $\varepsilon_k$-net in the set $C$
with $\mathcal{N}(C,\varepsilon_k)$ elements.
We define $C_{\infty} \coloneqq \cup_{k=0}^\infty C_k$.
The set $C_0$ contains only one point.
We choose this point to be $x_0$.

Fix $k \in \mathbb{N} \cup \{0\}$.
Following \cite[p.\ 94]{Buldygin2000}, we define a mapping
$\pi_k \colon C \to C_k$
such that $\pi_k(x)  = x$ if $x\in C_k$
and $\pi_k(x)$ equals a point in $C_k$
closest to $x$ if $x\not\in C_k$.

Let $x \in  C_{\infty}$. Then there exists
$ K(x) \in \mathbb{N} \cup \{0\}$ such that $x \in C_{K(x)}$.
We define $x_{K(x)} \coloneqq x$, and
$x_{k-1} \coloneqq \pi_{k-1}(x_k)$,
for
$k=1, \ldots,K(x)$. We have
\begin{align*}
G(x,\xi) = G(x_0,\xi)
+ \sum_{k=1}^{K(x)}(G(x_k,\xi)- G(x_{k-1},\xi)).
\end{align*}
Hence
\begin{align*}
\|G(x,\xi)\|_{Y_2}
\leq \|G(x_0,\xi)\|_{Y_2}
+
\sum_{k=1}^{K(x)}
\max_{x \in C_k}\|G(x,\xi)- G(\pi_{k-1}(x),\xi)\|_{Y_2}.
\end{align*}

Combined with \eqref{eq:mappingsubgaussian"},
$d_{Y_1}(x, \pi_{k-1}(x)) \leq \varepsilon_{k-1}$
for all $x \in C_k$,
and  \Cref{lem:meanmax}, we find that
for $k = 1, 2, \ldots$,
\begin{align*}
\mathbb{E}[\max_{x \in C_k}\,
\|G(x,\xi)-G(\pi_{k-1}(x),\xi)\|_{Y_2}]
\leq \varepsilon_{k-1}
\sqrt{2\ln(2\mathcal{N}\big(C, \varepsilon_k\big))}.
\end{align*}
We obtain
\begin{align*}
\mathbb{E}[\sup_{x \in C_{\infty}}\, \|G(x,\xi)\|_{Y_2}]
& \leq
\mathbb{E}[\|G(x_0,\xi)\|_{Y_2}]
+
\sum_{k=1}^{\infty}\,
\varepsilon_{k-1}
\sqrt{2\ln(2\mathcal{N}\big(C, \varepsilon_k\big))}.
\end{align*}
Since the set $C_{\infty}$
is dense in $C$,  $G$ is a Carath\'eodory mapping, and
\begin{align*}
\varepsilon_{k-1}
\sqrt{2\ln(2\mathcal{N}\big(C, \varepsilon_k\big))}
\leq 4 \int_{\varepsilon_{k+1}}^{\varepsilon_k}
\sqrt{2\ln(2\mathcal{N}\big(C, \varepsilon\big))}
\, \mathrm{d} \varepsilon,
\end{align*}
we obtain the assertion.
\end{proof}

The next lemma provides
an exponential moment inequality.
It is a direct consequence of Theorem~3 in \cite{Pinelis1986}
and Lemma~1 in \cite{Milz2021}, and has been established
in \cite{Milz2022b}.

\begin{lemma}[{see \cite[Lem.\ B.4]{Milz2022b}}]
\label{prop:saa:2020-11-21T20:25:01.71}
Let $\tau > 0$,
let $H$ be a real, separable Hilbert space, and let
$Z_i : \Theta\to H$
be independent, mean-zero random vectors such that
$\mathbb{E}[\cosh(\lambda \|Z_i\|_H)]\leq \exp(\lambda^2\tau^2/2)$
for all $\lambda \geq 0$
and $i \in \{1, 2, \ldots, N \}$. Then
for each
$\lambda \geq 0$,
$$
\mathbb{E}[\cosh(\lambda \| Z_1 + \cdots + Z_N\|_H)]
\leq \exp(3\lambda^2 \tau^2 N/4).
$$
\end{lemma}

Now we are ready to establish \Cref{prop:uniformexponentialtailboundaverage}.

\begin{proof}[{Proof of
 \Cref{prop:uniformexponentialtailboundaverage}}]
Defining the metric $d_{Y_1}(x,y) \coloneqq \sqrt{3/2}(M/\sqrt{N}) \|x-y\|_H$,
\Cref{thm:chaininghilbertspace,%
prop:saa:2020-11-21T20:25:01.71} imply
the assertion.
\end{proof}

The following example demonstrates an upper bound on
the integral in \eqref{eq:dudley-bound} under a bound on the
covering numbers. We use it
in the main text to establish mean convergence rates for
SAA optimal values
and criticality measures.

\begin{example}
\label{example:entropyintegralbound}
Let $c> 0$ be a constant,
and let $C \subset Y$ be a nonempty subset of a metric space
$(Y, d_{Y})$ with diameter $D > 0$.
If $\mathcal{N}(C, \varepsilon) \leq 2^{c/\varepsilon}$
for all $\varepsilon \in (0, D]$, then
\begin{align*}
\int_{0}^{D/2}
\sqrt{\ln(2\mathcal{N}(C, \varepsilon))}
\, \mathrm{d}\varepsilon
\leq
\int_{0}^{D/2}
(1+\sqrt{c/\varepsilon})
\, \mathrm{d}\varepsilon
\href{https://wolfr.am/1vpXtnRsS}{=} \sqrt{2cD}+D/2.
\end{align*}
\end{example}

\begin{footnotesize}
\bibliography{EnsembleControl-arXiv-v2}
\end{footnotesize}

\end{document}